\documentclass[11pt]{amsart}
\usepackage{amsmath,amssymb,amsthm}
\usepackage{booktabs}
\usepackage[margin=1.05in]{geometry}
\usepackage[hidelinks]{hyperref}

\newtheorem{theorem}{Theorem}
\newtheorem{corollary}{Corollary}
\newtheorem{lemma}{Lemma}
\theoremstyle{remark}
\newtheorem{remark}{Remark}

\newcommand{\Z}{\mathbb{Z}}
\newcommand{\Q}{\mathbb{Q}}
\newcommand{\R}{\mathbb{R}}
\newcommand{\ord}{\operatorname{ord}}
\newcommand{\den}{\operatorname{den}}
\newcommand{\fl}[1]{\lfloor #1\rfloor}
\newcommand{\ha}{\hat a}
\newcommand{\hb}{\hat b}
\newcommand{\hal}{\hat\alpha}
\newcommand{\hbe}{\hat\beta}
\newcommand{\one}{\mathbf 1}

\title{A new upper bound for the irrationality exponent of $\zeta(2)$}
\author{David Niedbala Giraudin}
\address{Independent researcher, Meaux, France}
\thanks{ORCID: 0009-0009-1526-1178}
\subjclass[2020]{Primary 11J82; Secondary 11Y60, 33C20, 33C60}
\keywords{irrationality exponent, $\zeta(2)$, $\pi^2$, Barnes-type integrals, Whipple's transformation, ball arithmetic}

\begin{document}

\begin{abstract}
We prove that the irrationality exponent of $\zeta(2)=\pi^2/6$ satisfies $\mu(\zeta(2))<5.0193784$.
The best published bound is $5.095412$ (Zudilin, 2014); a bound $5.0495243$ with a machine-checked
proof was made public in September 2026 by J.~Kleid. We use Zudilin's two hypergeometric
constructions at a new choice of parameters. The two families of linear forms in $1$ and
$\zeta(2)$ so obtained coincide for every sufficiently large index: this coincidence, an instance
of an identity conjectured by Zudilin, is established by comparing an arithmetic bound for the
denominators of the difference with an analytic bound for its size, without creative telescoping.
All numerical constants are certified in ball arithmetic.
\end{abstract}

\maketitle

\section{Introduction}

The irrationality exponent $\mu(\xi)$ of a real irrational number $\xi$ is the supremum of the
exponents $\mu$ for which $|\xi-p/q|<q^{-\mu}$ has infinitely many rational solutions $p/q$.
Since Ap\'ery's proof of the irrationality of $\zeta(2)$ \cite{Ap}, upper bounds for
$\mu(\zeta(2))=\mu(\pi^2)$ have been improved several times, the last published records being
$5.441243$ by Rhin and Viola \cite{RV} and $5.095412$ by Zudilin \cite{Z14}. In September 2026
J.~Kleid made public a machine-checked (Lean~4) proof of $\mu(\zeta(2))\le 5.0495243$ \cite{Kl},
obtained with Zudilin's construction at another choice of parameters.

\begin{theorem}\label{thm:main}
$\mu(\zeta(2))\le \dfrac{C_0+C_1}{C_0-\widetilde C_2}=5.0193783774690093995\ldots<5.0193784$,
with the constants of Table~\ref{tab:cert}.
\end{theorem}

\begin{corollary}
$\mu(\pi^2)<5.0193784$, and $\mu(\pi\sqrt d)<10.0387568$ for every positive rational $d$.
\end{corollary}

Indeed $\mu(\pi^2)=\mu(\zeta(2))$, and if $|\pi\sqrt d-p/q|<q^{-\mu}$ then $|d\pi^2-p^2/q^2|\ll(q^2)^{-\mu/2}$, so that $\mu(\pi\sqrt d)\le2\mu(\pi^2)$ (cf.\ \cite[\S1]{Z14}).

The construction is the one of \cite{Z14}: a Barnes-type integral $r_n=q_n\zeta(2)-p_n$ (the
\emph{first construction}) and a second integral $\hat r_n=\hat q_n\zeta(2)-\hat p_n$ built on a
very-well-poised pattern (the \emph{second construction}). Each one comes with its own $p$-adic
information on its coefficients; when the two linear forms coincide, both pieces of
information apply to the same numbers. Zudilin proved the coincidence at his parameters by
creative telescoping and conjectured it in general \cite[(19)]{Z14}, \cite{Z17}. At the
parameters used here we prove it for all large $n$ by a different argument
(Theorem~\ref{thm:coinc}): by Whipple's transformation $q_n+\hat q_n=0$, so that
$\Delta_n=r_n+\hat r_n$ is a rational number; its denominator is at most $e^{(K+o(1))n}$, while
$|\Delta_n|\le e^{(-K-1+o(1))n}$; hence $\Delta_n=0$.

\section{Zudilin's two constructions}\label{sec:constr}

We recall the facts we need from \cite{Z14}, keeping its notation. For a prime $p$, $\ord_p$
is the $p$-adic valuation; $D_m=\operatorname{lcm}(1,\dots,m)$; $\psi=\Gamma'/\Gamma$;
$(x)_m=x(x+1)\cdots(x+m-1)$.

\subsection{First construction}
Let $\alpha_1,\dots,\alpha_4,\beta_1,\dots,\beta_4$ be integers with
$\beta_1,\beta_2,\beta_3<\alpha_1,\dots,\alpha_4<\beta_4$ and
$\alpha_1+\dots+\alpha_4>\beta_1+\dots+\beta_4$. For $n\ge0$ put $a_j=\alpha_jn+1$
($1\le j\le4$), $b_j=\beta_jn+1$ ($1\le j\le3$), $b_4=\beta_4n+2$, $d=\sum_j a_j-\sum_j b_j$,
\[
R(t)=\prod_{j=1}^{3}\frac{(t+b_j)_{a_j-b_j}}{(a_j-b_j)!}\cdot\frac{(b_4-a_4-1)!}{(t+a_4)_{b_4-a_4}},
\qquad
r_n=\frac{(-1)^d}{2\pi i}\int_{C-i\infty}^{C+i\infty}\Bigl(\frac{\pi}{\sin\pi t}\Bigr)^2R(t)\,dt .
\]
By \cite[Prop.~1]{Z14}, $r_n=q_n\zeta(2)-p_n$ with $q_n\in\Z$; explicitly
$q_n=(-1)^d\sum_{k=a_4^*}^{b_4-1}C_k$ with
$C_k=(-1)^{d+b_4+k}\binom{k-b_1}{k-a_1}\binom{k-b_2}{k-a_2}\binom{k-b_3}{k-a_3}\binom{b_4-a_4-1}{k-a_4}$,
where $a_1^*\le\dots\le a_4^*$ is the ordering of the $a_j$ (and $\alpha_1^*\le\dots\le\alpha_4^*$ that of the $\alpha_j$); in the definition of $r_n$, $-a_2^*<C<1-b_2^*$, where $b_1^*\le b_2^*\le b_3^*$ orders $b_1,b_2,b_3$ \cite[(10)]{Z14}.
Let $\tau_1$ be the real zero and $\tau_0$, $\bar\tau_0$ the non-real zeros of the cubic
polynomial $\prod_j(\tau-\alpha_j)-\prod_j(\tau-\beta_j)$, and
\[
f_0(\tau)=\sum_{j=1}^4\bigl(\alpha_j\log(\tau-\alpha_j)-\beta_j\log(\tau-\beta_j)\bigr)
-\sum_{j=1}^3(\alpha_j-\beta_j)\log(\alpha_j-\beta_j)+(\beta_4-\alpha_4)\log(\beta_4-\alpha_4).
\]
By \cite[Lemma~6]{Z14},
$\limsup_n\frac1n\log|r_n|=-C_0$ and $\lim_n\frac1n\log|q_n|=C_1$, where
$C_0=-\operatorname{Re}f_0(\tau_0)$ and $C_1=\operatorname{Re}f_0(\tau_1)$.
Let $\gamma_1\ge\gamma_2$ be the two integers
$\max\{\alpha_1-\beta_1,\alpha_2-\beta_2,\alpha_3-\beta_3,\beta_4-\alpha_2^*\}$ and
$\max\{\sum\alpha_j-\sum\beta_j,\beta_4-\alpha_2^*\}$, and $\gamma_0=\max_{j,k}|\alpha_j-\beta_k|$.
By \cite[Lemma~7]{Z14},
\begin{equation}\label{eq:lemma7}
\Phi_n^{-1}q_n\in\Z,\qquad \Phi_n^{-1}D_{\gamma_1n}D_{\gamma_2n}\,p_n\in\Z,\qquad
\Phi_n=\prod_{\sqrt{\gamma_0n}<p\le\gamma_2n}p^{\varphi(n/p)},
\end{equation}
with the $1$-periodic integer-valued function
$\varphi(x)=\max_{\sigma}\bigl(\fl{(\beta_4-\alpha_4)x}-\fl{(\beta_4-\alpha'_4)x}
-\sum_{j=1}^3(\fl{(\alpha_j-\beta_j)x}-\fl{(\alpha'_j-\beta_j)x})\bigr)$, the maximum
being over the permutations $\alpha'=\sigma\alpha$ of $(\alpha_1,\dots,\alpha_4)$.

\subsection{Second construction}
Let $\hal_0,\dots,\hbe_3$ be integers with
$\tfrac12\hbe_0,\hbe_1<\tfrac12\hal_0,\hal_1,\hal_2,\hal_3<\hbe_2,\hbe_3$ and
$\sum\hal_j=\sum\hbe_j$. Put $\ha_0=\hal_0n+2$, $\ha_j=\hal_jn+1$ ($j=1,2,3$),
$\hb_0=\hbe_0n+2$, $\hb_1=\hbe_1n+1$, $\hb_j=\hbe_jn+2$ ($j=2,3$), $\hat d=\hb_2+\hb_3$,
\[
\hat R(t)=\frac{(2t+\hb_0)_{\ha_0-\hb_0}}{(\ha_0-\hb_0)!}\cdot\frac{(t+\hb_1)_{\ha_1-\hb_1}}{(\ha_1-\hb_1)!}
\cdot\prod_{j=2}^3\frac{(\hb_j-\ha_j-1)!}{(t+\ha_j)_{\hb_j-\ha_j}},
\qquad
\hat r_n=\frac{(-1)^{\hat d}}{2\pi i}\int_{C/2-i\infty}^{C/2+i\infty}\frac{\pi}{\sin2\pi t}\hat R(t)\,dt,
\]
with $-\ha_0^*<C<1-\hb_0$, where $\ha_1^*\le\ha_2^*\le\ha_3^*$ orders $\ha_1,\ha_2,\ha_3$,
$\hb_2^*\le\hb_3^*$ orders $\hb_2,\hb_3$, and $\ha_0^*=\min\{\ha_0,2\ha_2^*\}$; the orderings $\hal_j^*$, $\hbe_j^*$ of the $\hal_j$, $\hbe_j$ are defined in the same way. Writing
$\hat R(t)=\sum_{k=\ha_3^*}^{\hb_2^*-1}A_k(t+k)^{-2}+\sum_{k=\ha_2^*}^{\hb_3^*-1}B_k(t+k)^{-1}$,
the proof of \cite[Prop.~3]{Z14} gives $\hat r_n=\hat q_n\zeta(2)-\hat p_n$ with
\begin{equation}\label{eq:phat}
\hat q_n=(-1)^{\hat d}\sum_kA_k\in\Z,\qquad
\hat p_n=(-1)^{\hat d}\Bigl(\sum_{k}2A_k\Lambda_2(2k-\ha_0^*)+\sum_{k}B_k\Lambda_1(2k-\ha_0^*)\Bigr),
\end{equation}
where $\Lambda_s(L)=\sum_{\ell=1}^{L}(-1)^{\ell-1}\ell^{-s}$, and
$D_{\hat c_1n}D_{\hat c_2n}\hat p_n\in\Z$ for explicit integers $\hat c_1\ge\hat c_2$
\cite[(25)]{Z14}. For every prime $p$ and every $k$ \cite[(24)]{Z14},
\begin{equation}\label{eq:24}
\ord_pA_k\ge E_k(p),\qquad 1+\ord_pB_k\ge E_k(p),
\end{equation}
where $E_k(p)$ is the following sum of integer parts:
\[
\Bigl\lfloor\tfrac{2k-\hb_0}{p}\Bigr\rfloor-\Bigl\lfloor\tfrac{2k-\ha_0}{p}\Bigr\rfloor-\Bigl\lfloor\tfrac{\ha_0-\hb_0}{p}\Bigr\rfloor
+\Bigl\lfloor\tfrac{k-\hb_1}{p}\Bigr\rfloor-\Bigl\lfloor\tfrac{k-\ha_1}{p}\Bigr\rfloor-\Bigl\lfloor\tfrac{\ha_1-\hb_1}{p}\Bigr\rfloor
+\sum_{j=2}^3\Bigl(\Bigl\lfloor\tfrac{\hb_j-\ha_j-1}{p}\Bigr\rfloor-\Bigl\lfloor\tfrac{k-\ha_j}{p}\Bigr\rfloor-\Bigl\lfloor\tfrac{\hb_j-1-k}{p}\Bigr\rfloor\Bigr).
\]
By \cite[Lemma~8]{Z14}, $\ord_p\hat q_n\ge\hat\varphi(n/p)$ and
$\ord_p\hat p_n\ge-2+\hat\varphi(n/p)$ for the primes $p>\sqrt{\hat c_1n}$, where
$\hat\varphi(x)=\min_{0\le y<1}E(x,y)$ and
\[
\begin{aligned}
E(x,y)={}&\fl{2y-\hbe_0x}-\fl{2y-\hal_0x}-\fl{(\hal_0-\hbe_0)x}+\fl{y-\hbe_1x}-\fl{y-\hal_1x}-\fl{(\hal_1-\hbe_1)x}\\
&+\textstyle\sum_{j=2}^3\bigl(\fl{(\hbe_j-\hal_j)x}-\fl{\hbe_jx-y}-\fl{y-\hal_jx}\bigr).
\end{aligned}
\]

\subsection{The parameters}\label{sec:param}
Throughout, except in Section~\ref{sec:rem},
\begin{equation}\label{eq:param}
(\alpha_1,\alpha_2,\alpha_3,\alpha_4;\,\beta_1,\beta_2,\beta_3,\beta_4)=(146,123,100,169;\,0,23,46,292),
\end{equation}
so that $\beta_1=0$, $\beta_2=\alpha_4-\alpha_1$, $\beta_3=\alpha_4-\alpha_2$, and the second
construction is taken with
\begin{equation}\label{eq:hatparam}
(\hal_0;\hal_1,\hal_2,\hal_3\,|\,\hbe_0;\hbe_1,\hbe_2,\hbe_3)=(223;54,77,100\,|\,31;0,200,223),
\end{equation}
i.e.\ $\hal_0=\hbe_3=\beta_4+\alpha_3-\alpha_4$, $\hbe_0=2\alpha_3-\alpha_4$,
$\hal_1=\alpha_2+\alpha_3-\alpha_4$, $\hal_2=\alpha_1+\alpha_3-\alpha_4$, $\hal_3=\alpha_3$,
$\hbe_2=\alpha_1+\alpha_2+\alpha_3-\alpha_4$: this is the correspondence of \cite[\S\S5--6]{Z14}
(identity (19) there, after the shift $t\mapsto t-(\alpha_4-\alpha_3)n$). All the conditions above
hold, and $(\gamma_1,\gamma_2)=(177,169)$, $\gamma_0=192$, $(\hat c_1,\hat c_2)=(292,246)$.

\section{Four lemmas}

\begin{lemma}[Whipple]\label{lem:whipple}
For every $n\ge0$, $q_n+\hat q_n=0$.
\end{lemma}

\begin{proof}
Under the conditions $\beta_1=0$, $\beta_2=\alpha_4-\alpha_1$, $\beta_3=\alpha_4-\alpha_2$, the
finite sums $q_n$ and $\hat q_n$ are, up to sign, the terminating ${}_4F_3(1)$ and
${}_5F_4(1)$ series on the two sides of Whipple's transformation
\cite[(2.4.2.3)]{Sl}; this identification is \cite[Remark~5]{Z14}. With the normalisations of
Section~\ref{sec:constr} the identity reads $q_n=-\hat q_n$. (This sign is confirmed in exact
arithmetic for $0\le n\le10$; see the ancillary file \texttt{verify\_\allowbreak zeta2\_\allowbreak exact.py}.)
\end{proof}

\begin{lemma}\label{lem:refined}
Let $\ha_3^*\le k<\hb_2^*$ with $A_k\ne0$, and let $\mathcal N_k$ be the set of nonzero integers
$i-2k$ ($\hb_0\le i<\ha_0$), $i-k$ ($\hb_1\le i<\ha_1$), $i-k$ ($\ha_2\le i<\hb_2$, $i\ne k$),
$i-k$ ($\ha_3\le i<\hb_3$, $i\ne k$). Let $p$ be an odd prime with $p^2>\max_{c\in\mathcal N_k}|c|$,
and $\delta_k(p)=1$ if $p$ divides an element of $\mathcal N_k$, $\delta_k(p)=0$ otherwise. Then
\[
\ord_pB_k\ge\ord_pA_k-\delta_k(p).
\]
\end{lemma}

\begin{proof}
Since $\ha_3^*\le k<\hb_2^*$, the factor $t+k$ occurs once in each of the two products of the
denominator of $\hat R$, and $G(t)=\hat R(t)(t+k)^2$ is a product of linear factors with
$G(-k)=A_k$ and $B_k=G'(-k)$. As $A_k\ne0$, no factor of the numerator vanishes at $t=-k$, and
\[
\frac{B_k}{A_k}=\frac{G'}{G}(-k)=\sum_{\hb_0\le i<\ha_0}\frac{2}{i-2k}+\sum_{\hb_1\le i<\ha_1}\frac1{i-k}
-\sum_{\substack{\ha_2\le i<\hb_2\\ i\ne k}}\frac1{i-k}-\sum_{\substack{\ha_3\le i<\hb_3\\ i\ne k}}\frac1{i-k}.
\]
Every denominator lies in $\mathcal N_k$; by the assumption on $p$ none is divisible by $p^2$,
and none is divisible by $p$ unless $\delta_k(p)=1$. Hence $\ord_p(B_k/A_k)\ge-\delta_k(p)$.
\end{proof}

For $x\in(0,1]$ and $y\in\R$ let $\chi(x,y)=\one[2y-\hal_0^*x\ge1]$ with
$\hal_0^*=\min\{\hal_0,2\hal_2^*\}$, and let $\delta(x,y)=1$ if $\hbe_0x\le2y<\hal_0x$, or
$\hbe_1x\le y<\hal_1x$, or one of the intervals $[\hbe_0x-2y,\hal_0x-2y)$,
$[\hbe_1x-y,\hal_1x-y)$, $[\hal_2x-y,\hbe_2x-y)$, $[\hal_3x-y,\hbe_3x-y)$ contains a nonzero
integer; $\delta(x,y)=0$ otherwise. With $J_1=[\hal_3^*x,\hbe_2^*x)$ and
$J_2=[\hal_2^*x,\hbe_3^*x)$ define
\[
m(x)=\min\Bigl\{\inf_{y\in J_1}\bigl(E-2\chi\bigr),\ \inf_{y\in J_2\setminus J_1}\bigl(E-1-\chi\bigr),\
\inf_{y\in J_1}\bigl(E-\delta-\chi\bigr)\Bigr\}(x,y),\qquad \upsilon(x)=\max\{0,-m(x)\}.
\]
Put $e_1(x)=\one[x\ge1/\gamma_1]+\one[x\ge1/\gamma_2]$,
$e_2(x)=\one[x\ge1/\hat c_1]+\one[x\ge1/\hat c_2]$,
$u_1(x)=e_1(x)-\varphi(x)\one[x\ge1/\gamma_2]$, $u_2(x)=\min\{e_2(x),\upsilon(x)\}$, and
\begin{equation}\label{eq:K}
K=\int_0^1\max\{u_1(x),u_2(x)\}\frac{dx}{x^2}+2-\int_0^1\min\{\varphi(\theta),\hat\varphi(\theta)\}\,d\psi(1+\theta).
\end{equation}
All the functions involved are step functions with rational jumps; $K$ is a finite sum of
values of $1/x$ and of $\psi$ at rational points.

\begin{lemma}\label{lem:den}
Let $\den(\xi)$ denote the denominator of $\xi\in\Q$. Then
$\limsup_n\frac1n\log\den(p_n+\hat p_n)\le K$.
\end{lemma}

\begin{proof}
Let $c\ge\max\{\gamma_1,\hat c_1\}$ be a constant, depending only on the parameters, such that for $p>\sqrt{cn}$ every integer $L$ occurring as an argument of $\Lambda_s$ in \eqref{eq:phat} and every element of every set $\mathcal N_k$ of Lemma~\ref{lem:refined} is smaller than $p^2$ in absolute value. Since $D_{\gamma_1n}D_{\gamma_2n}p_n$ and
$D_{\hat c_1n}D_{\hat c_2n}\hat p_n$ are integers, the primes $p\le\sqrt{cn}$ contribute
$O(\sqrt n\log n)$ to $\log\den(p_n+\hat p_n)$. Let $p>\sqrt{cn}$ and $x=n/p$; then
$\ord_pD_{\gamma n}=\one[p\le\gamma n]$ and
$\ord_p(p_n+\hat p_n)\ge\min\{\ord_pp_n,\ord_p\hat p_n\}$.

\emph{Primes $p>n$.} By \eqref{eq:lemma7}, $\ord_pp_n\ge-u_1(x)$. By \cite[(25)]{Z14},
$\ord_p\hat p_n\ge-e_2(x)$. In \eqref{eq:phat} we bound the terms one by one: $\ord_pA_k\ge
E_k(p)$ and, by \eqref{eq:24} and Lemma~\ref{lem:refined}, $\ord_pB_k\ge E_k(p)-\delta_k(p)$
when $\ha_3^*\le k<\hb_2^*$ and $A_k\ne0$, $\ord_pB_k\ge E_k(p)-1$ otherwise; moreover
$\ord_p\Lambda_2(L)\ge-2\cdot\one[p\le L]$ and $\ord_p\Lambda_1(L)\ge-\one[p\le L]$ since $p^2>L$.
Setting $y=(k-1)/p$ as in the proof of \cite[Lemma~8]{Z14}, these bounds become
$E-2\chi$, $E-1-\chi$ and $E-\delta-\chi$ at $(x,y)$ (the $A$-terms live on $J_1$, the
$B$-terms on $J_2$), up to a set of primes on which $x$ lies within $O(1/p)$ of one of finitely
many rational points; these contribute $o(n)$. Hence $\ord_p\hat p_n\ge-\min\{e_2(x),\upsilon(x)\}=-u_2(x)$
outside that set, and $\ord_p(p_n+\hat p_n)\ge-\max\{u_1,u_2\}(x)$.

\emph{Primes $\sqrt{cn}<p\le n$.} Here $e_1=e_2=2$ and, by \eqref{eq:lemma7} and
\cite[Lemma~8]{Z14}, $\ord_p(p_n+\hat p_n)\ge-2+\min\{\varphi,\hat\varphi\}(n/p)$.

By the prime number theorem, for a step function $w$ vanishing near $0$,
$\sum_{p>n}w(n/p)\log p=n\int_0^1w(x)\,dx/x^2+o(n)$ and, for $1$-periodic $w$,
$\sum_{\sqrt{cn}<p\le n}w(n/p)\log p=n\int_0^1w(\theta)\,d\psi(1+\theta)+o(n)$, since
$\psi'(1+\theta)=\sum_{k\ge1}(k+\theta)^{-2}$ (this is the standard computation behind \cite[Lemma~7]{Z14}). With $\psi(2)-\psi(1)=1$ this
gives the bound $K$.
\end{proof}

Let
\[
\begin{aligned}
\Sigma=\{&(+1,2,\hal_0),(+1,1,\hal_1),(+1,1,\hal_2),(+1,1,\hal_3),\\
&(-1,1,\hbe_2),(-1,1,\hbe_3),(+1,-2,-\hbe_0),(+1,-1,-\hbe_1)\},
\end{aligned}
\]
and for $\sigma_0\in\R$ and $y\in\R$ let
\[
\hat H(y)=\sum_{(\epsilon,k,c)\in\Sigma}\epsilon\operatorname{Re}\bigl(w\log w\bigr)\big|_{w=k(\sigma_0+iy)+c}
+\pi|y|+h_\Pi,
\]
\[
h_\Pi=(\hbe_2-\hal_2)\log(\hbe_2-\hal_2)+(\hbe_3-\hal_3)\log(\hbe_3-\hal_3)-(\hal_0-\hbe_0)\log(\hal_0-\hbe_0)-(\hal_1-\hbe_1)\log(\hal_1-\hbe_1).
\]

\begin{lemma}\label{lem:decay}
For $\sigma_0=-32.2134$ one has $\limsup_n\frac1n\log|\hat r_n|\le\sup_{y\in\R}\hat H(y)\le-299.1530<-K-1$.
\end{lemma}

\begin{proof}
The interval $(-\hal_0^*/2,-\hbe_0/2)=(-77,-15.5)$ contains $\sigma_0$, so for large $n$ there is
$c_n\in\frac12\Z+\frac14$ with $|c_n-n\sigma_0|\le\frac14$ inside the strip
$-\ha_0^*/2<\operatorname{Re}t<(1-\hb_0)/2$, in which the integrand of $\hat r_n$ is holomorphic
(the poles of $\pi/\sin2\pi t$ there are cancelled by zeros of $\hat R$, \cite[\S6]{Z14}); we
move the contour to $\operatorname{Re}t=c_n$. On it, $|\pi/\sin2\pi t|=\pi/\cosh(2\pi\operatorname{Im}t)$.
On this line $\operatorname{Re}(2t+\hb_0)$ and $\operatorname{Re}(t+\hb_1)$ are negative; we use
$1/\Gamma(z)=\Gamma(1-z)\sin(\pi z)/\pi$ and $|\sin\pi z|\le e^{\pi|\operatorname{Im}z|}$ for
these two factors. The eight Gamma factors then have arguments $nw+O(1)$, $w=k\tau+c$,
$(\epsilon,k,c)\in\Sigma$, $\tau=\sigma_0+iy$, $y=\operatorname{Im}t/n$, all with
$\operatorname{Re}w\ge21.78$. Stirling's formula, applied to them and to the four factorials of
the normalising constant, gives
$\log|\hat R(t)\pi/\sin2\pi t|\le n\hat H(y)+O(\log(n+|\operatorname{Im}t|))$ uniformly on the
line: the terms in $n\log n$ and in $n$ cancel because $\sum\hal_j=\sum\hbe_j$, and the
exponential factors give $e^{\pi n(2|y|+|y|-2|y|)}=e^{\pi n|y|}$. For $|y|$ large,
$\hat H(y)=-2\pi|y|+(\hal_0-\hbe_0)\log2+h_\Pi+O(y^{-2})$, so the integral over the line is
$\le e^{n\sup\hat H+O(\log n)}$. The numerical bound is certified (file \texttt{verify\_\allowbreak zeta2.py}):
$\hat H$ is even; on $[0,2000]$ it is enclosed by $47$ Arb balls, the largest upper bound being
$-299.15302$; for $y\ge2000$,
$\hat H(y)\le-2\pi y+(\hal_0-\hbe_0)\log2+h_\Pi+\sum_{\Sigma}\frac{(\operatorname{Re}w)^3}{2k^2y^2}<-12472$.
\end{proof}

\begin{remark}
At the non-real critical point $\hat\tau_0$ of the second construction (a root of
$(2\tau+\hal_0)^2\prod_{j=1}^3(\tau+\hal_j)-(2\tau+\hbe_0)^2\prod_{j=1}^3(\tau+\hbe_j)$,
$\hat\tau_0=-32.21344848\ldots\pm14.37727696\ldots i$), taking $\sigma_0=\operatorname{Re}\hat\tau_0$ one finds $\hat H(\operatorname{Im}\hat\tau_0)=-325.4478457713\ldots=-C_0$ to all computed
digits, as Theorem~\ref{thm:coinc} predicts.
\end{remark}

\section{Coincidence of the two families}

\begin{theorem}\label{thm:coinc}
With the parameters \eqref{eq:param}--\eqref{eq:hatparam}, $r_n+\hat r_n=0$ for all
sufficiently large $n$; that is, $q_n=-\hat q_n$ and $p_n=-\hat p_n$.
\end{theorem}

\begin{proof}
By Lemma~\ref{lem:whipple}, $\Delta_n:=r_n+\hat r_n=-(p_n+\hat p_n)\in\Q$. By
\cite[Lemma~6]{Z14} and Lemma~\ref{lem:decay},
$|\Delta_n|\le|r_n|+|\hat r_n|\le e^{(-\min\{C_0,K+1\}+o(1))n}$, and by Lemma~\ref{lem:den},
$\den(\Delta_n)\le e^{(K+o(1))n}$. Since $C_0-K=31.4070\ldots>0$ (Table~\ref{tab:cert}),
$|\den(\Delta_n)\Delta_n|<1$ for large $n$; this is an integer, so $\Delta_n=0$.
\end{proof}

For $0\le n\le2$ the identity $r_n+\hat r_n=0$ is also checked in exact rational arithmetic
(\texttt{verify\_\allowbreak zeta2\_\allowbreak exact.py}).

\section{Proof of Theorem~\ref{thm:main}}

Let $n$ be large and $c'=\max\{\gamma_0,\hat c_1\}$. By Theorem~\ref{thm:coinc} and Lemma~\ref{lem:whipple}, for the primes
$\sqrt{c'n}<p\le\gamma_2n$ we have, from \eqref{eq:lemma7} and \cite[Lemma~8]{Z14},
$\ord_pq_n\ge\tilde\varphi(n/p)$ and $\ord_pp_n\ge-2+\tilde\varphi(n/p)$, where
$\tilde\varphi=\max\{\varphi,\hat\varphi\}$. Hence, with
$\widetilde\Phi_n=\prod_{\sqrt{c'n}<p\le\gamma_2n}p^{\tilde\varphi(n/p)}$ (the primes $p\le\sqrt{c'n}$ do not affect the limits below),
\[
\widetilde\Phi_n^{-1}q_n\in\Z,\qquad\widetilde\Phi_n^{-1}D_{\gamma_1n}D_{\gamma_2n}\,p_n\in\Z,
\]
and, by the prime number theorem,
\[
\lim_n\frac1n\log\bigl(\widetilde\Phi_n^{-1}D_{\gamma_1n}D_{\gamma_2n}\bigr)
=\gamma_1+\gamma_2-\int_0^1\tilde\varphi(x)\,d\psi(x)+\int_0^{1/\gamma_2}\tilde\varphi(x)\frac{dx}{x^2}
=\widetilde C_2 .
\]
Together with $\limsup\frac1n\log|r_n|=-C_0$, $\lim\frac1n\log|q_n|=C_1$ and
$C_0>\widetilde C_2$, \cite[Lemma~2.1]{Ha} gives
$\mu(\zeta(2))\le(C_0+C_1)/(C_0-\widetilde C_2)$, which is the value in Table~\ref{tab:cert}.
\qed

\section{Certified values}\label{sec:cert}

The step functions $\varphi$, $\hat\varphi$, $\upsilon$ are computed in exact rational
arithmetic ($\upsilon$ in exact integer arithmetic after scaling), on the grid of all
$k/c$ with $c$ in an explicit finite set of denominators; each cell of the grid used for $K$ is
tested at three interior points, and none is found non-constant. The transcendental
quantities ($\tau_0$, $\tau_1$ via certified polynomial roots, logarithms, values of $\psi$ at
rational points, $\hat H$ on intervals) are computed with Arb \cite{Arb} at $220$ bits.

\begin{table}[h]
\centering
\begin{tabular}{@{}ll@{}}
\toprule
quantity & certified enclosure \\
\midrule
$C_0$ & $325.447845771316912\pm1.2\cdot10^{-16}$\\
$C_1$ & $467.694080399460208\pm2.3\cdot10^{-16}$\\
$\widetilde C_2$ & $167.431879146638171\pm2.9\cdot10^{-16}$\\
$\mu=(C_0+C_1)/(C_0-\widetilde C_2)$ & $5.0193783774690093995\pm5.4\cdot10^{-21}$\\
$K$ & $294.040748803464213\pm3.9\cdot10^{-16}$\\
$C_0-K$ & $31.407096967853\pm3.1\cdot10^{-13}$\\
$\sup_y\hat H(y)$, $\sigma_0=-32.2134$ & $\le-299.15302$\\
\bottomrule
\end{tabular}
\medskip
\caption{Certified constants for the parameters \eqref{eq:param}--\eqref{eq:hatparam}.}
\label{tab:cert}
\end{table}

\section{Remarks}\label{sec:rem}

\begin{remark}
For the smaller parameters $(19,16,13,22;\,0,3,6,38)$ the same proof gives
$\mu(\zeta(2))\le5.0204524708476832909\ldots$, with $C_0=42.33438299777\ldots$,
$K=38.31605710342\ldots$, $C_0-K=4.01832589434\ldots$ and $\sup\hat H\le-39.3547$ on
$\operatorname{Re}\tau=-4.1774$; this instance is quicker to verify and is also treated by the
ancillary files.
\end{remark}

\section*{Ancillary files}

\noindent\texttt{verify\_\allowbreak zeta2.py} (Python~3 with \texttt{python-flint} and \texttt{numpy})
recomputes and certifies every entry of Table~\ref{tab:cert}, and the corresponding values for the
smaller parameters; it starts with a self-test reproducing Zudilin's value $5.09541178\ldots$\
\texttt{verify\_\allowbreak zeta2\_\allowbreak exact.py} (standard library only) checks in exact rational arithmetic:
$q_n+\hat q_n=0$ for $n\le10$ (main parameters) and $n\le25$ (smaller ones);
$r_n+\hat r_n=0$, i.e.\ also $p_n+\hat p_n=0$, for $n\le2$ and $n\le6$; and the term-by-term
$p$-adic bounds used in Lemma~\ref{lem:den} against the true valuations ($11\,715$ and
$10\,820$ comparisons, no exception).

\section*{Acknowledgements}
The author thanks W.~Zudilin, whose constructions and lemmas this note rests on. Technical
assistance: Claude (Anthropic).


\begin{thebibliography}{99}
\bibitem{Ap} R.~Ap\'ery, Irrationalit\'e de $\zeta(2)$ et $\zeta(3)$,
\emph{Ast\'erisque} \textbf{61} (1979), 11--13.
\bibitem{Ha} M.~Hata, Rational approximations to $\pi$ and some other numbers,
\emph{Acta Arith.} \textbf{63} (1993), 335--349.
\bibitem{Arb} F.~Johansson, Arb: efficient arbitrary-precision midpoint-radius interval
arithmetic, \emph{IEEE Trans. Comput.} \textbf{66} (2017), 1281--1292.
\bibitem{Kl} J.~Kleid, $\mu(\zeta(2))\le5.0495243$, machine-checked in Lean~4, repository
\texttt{qedbook/zeta2-irrationality-measure}, GitHub, September 2026.
\bibitem{RV} G.~Rhin and C.~Viola, On a permutation group related to $\zeta(2)$,
\emph{Acta Arith.} \textbf{77} (1996), 23--56.
\bibitem{Sl} L.~J.~Slater, \emph{Generalized hypergeometric functions}, 2nd ed., Cambridge
University Press, 1966.
\bibitem{Z14} W.~Zudilin, Two hypergeometric tales and a new irrationality measure of
$\zeta(2)$, \emph{Ann. Math. Qu\'e.} \textbf{38} (2014), 101--117; arXiv:1310.1526.
\bibitem{Z17} W.~Zudilin, Hypergeometric heritage of W.~N.~Bailey, with an appendix: Bailey's
letters to F.~Dyson, arXiv:1611.08806.
\end{thebibliography}
\end{document}